\documentclass[a4paper,11pt]{article}
\usepackage[T1]{fontenc}
\usepackage{lmodern,amsmath,amsthm,amsfonts,amssymb,graphicx,float,microtype,thmtools,underscore,mathtools,xurl,thm-restate}
\usepackage[dvipsnames,svgnames,table]{xcolor}
\usepackage[labelsep = period, justification = centering]{caption}
\usepackage[shortlabels]{enumitem}
\setlist[itemize]{topsep=0ex,itemsep=0ex,parsep=0ex}
\setlist[enumerate]{topsep=0ex,itemsep=0ex,parsep=0ex}
\usepackage[unicode=true]{hyperref}
\hypersetup{
colorlinks,
breaklinks=true,
linkcolor={blue!60!black},
citecolor={black},
urlcolor={blue!60!black},
pdftitle={The Dominating 4-Colour Theorem}}
\usepackage[capitalise, compress, nameinlink, noabbrev]{cleveref}
\crefname{lem}{Lemma}{Lemmas}
\crefname{thm}{Theorem}{Theorems}
\crefname{cor}{Corollary}{Corollaries}
\newcommand{\defn}[1]{\textcolor{Maroon}{\emph{#1}}}
\newcommand{\mathdefn}[1]{\textcolor{Maroon}{#1}}
\usepackage[longnamesfirst,numbers,sort&compress]{natbib}
\makeatletter
\def\NAT@spacechar{~}
\makeatother
\usepackage[tmargin=30mm,bmargin=30mm,lmargin=30mm,rmargin=30mm]{geometry}
\renewcommand{\baselinestretch}{1.1}
\allowdisplaybreaks
\DeclarePairedDelimiter{\abs}{\lvert}{\rvert}
\DeclarePairedDelimiter{\set}{\lbrace}{\rbrace} 
\renewcommand{\epsilon}{\varepsilon}
\renewcommand{\emptyset}{\varnothing}
\renewcommand{\ge}{\geqslant}
\renewcommand{\le}{\leqslant}
\renewcommand{\geq}{\geqslant}
\renewcommand{\leq}{\leqslant}
\DeclareMathOperator{\ind}{index}

\newcommand{\CC}{\mathcal{C}}

\newcommand{\TT}{\mathcal{T}}
\renewcommand{\thefootnote}{\fnsymbol{footnote}}
\usepackage{amsthm,thmtools,thm-restate}

\declaretheoremstyle[
spaceabove = .2\baselineskip\@plus.2\baselineskip\@minus.2\baselineskip, 
spacebelow = .5\baselineskip\@plus.2\baselineskip\@minus.2\baselineskip,
headfont = \normalfont\itshape,
notefont = \mdseries, 
notebraces = {}{},
bodyfont = \normalfont,
postheadspace = .5em,
headpunct = .,
qed = \qedsymbol
]{proofstyle}

\declaretheorem[name = Theorem, numberwithin = section, style = plain]{thm}
\declaretheorem[name = Claim, numbered = no, style = plain]{claim}
\declaretheorem[name = Corollary, numberlike = thm, style = plain]{cor}

\declaretheorem[name = Lemma, numberlike = thm, style = plain]{lem}

\declaretheorem[name = Proof of claim, numbered = no, style = proofstyle, refname = {Proof,Proofs}, Refname = {Proof,Proofs}]{proofclaim}
\crefname{thm}{Theorem}{Theorems}
\crefname{lem}{Lemma}{Lemmas}
\crefname{obs}{Observation}{Observations}
\crefname{conj}{Conjecture}{Conjectures}
\crefname{claim}{Claim}{Claims}
\crefname{prob}{Problem}{Problems}
\crefname{prop}{Proposition}{Propositions}
\crefname{cor}{Corollary}{Corollaries}
\crefname{rmk}{Remark}{Remarks}
\begin{document}
\title{\bf\boldmath\fontsize{18pt}{20pt}\selectfont The Dominating 4-Colour Theorem}

\author{
\qquad Ant\'{o}nio Gir\~{a}o\,\footnotemark[1] 
\qquad Freddie Illingworth\,\footnotemark[1] 
\qquad Bojan Mohar\,\footnotemark[2]
\\
Sergey Norin\,\footnotemark[3] 
\qquad Raphael Steiner\,\footnotemark[4]
\qquad  Youri Tamitegama\,\footnotemark[5]
\\
Jane Tan\,\footnotemark[6]
\qquad David~R.~Wood\,\footnotemark[7]
\qquad Jung Hon Yip\,\footnotemark[7]
}

\maketitle

\begin{abstract}
A \defn{dominating $K_t$-model} in a graph $G$ is a sequence $(T_1,\dots,T_t)$ of pairwise vertex-disjoint connected subgraphs of $G$, such that whenever $1\leq i<j\leq t$ every vertex in $T_j$ has a neighbour in $T_i$. Replacing ``every vertex in $T_j$'' by ``some vertex in $T_j$'' retrieves the standard definition of $K_t$-model, which is equivalent to a $K_t$-minor in~$G$. We prove that every graph with no dominating $K_5$-model is $4$-colourable. This generalises and is significantly stronger than the 4-colour theorem for planar graphs or for graphs with no $K_5$-minor. It also makes progress towards Haj\'{o}s' conjecture on $K_5$-subdivisions in $5$-chromatic graphs. 
\end{abstract}

\footnotetext[1]{Department of Mathematics, University College London, UK (\textsf{\{\href{mailto:a.girao@ucl.ac.uk}{a.girao},\allowbreak\href{mailto:f.illingworth@ucl.ac.uk}{f.illingworth}\}\allowbreak @ucl.ac.uk}). Research of Illingworth supported by a Robert Bartnik visiting fellowship.}

\footnotetext[2]{Department of Mathematics, Simon Fraser University, Burnaby, BC, Canada (\textsf{\href{mailto:mohar@sfu.ca}{mohar@sfu.ca}}) and FMF, Department of Mathematics, University of Ljubljana, Ljubljana, Slovenia. Research supported by the NSERC Discovery Grant R832714 (Canada), by the ERC Synergy grant (European Union, ERC, KARST, project number 101071836), and by the Core Research Program P1-0297 of ARIS (Slovenia).}

\footnotetext[3]{Department of Mathematics and Statistics, McGill University, Montr\'eal, Canada (\textsf{\href{mailto:sergey.norin@mcgill.ca}{sergey.norin\allowbreak@mcgill.ca}}).  Research supported by NSERC.}

\footnotetext[4]{Department of Mathematics, ETH Zürich, Switzerland (\textsf{\href{mailto:raphaelmario.steiner@math.ethz.ch}{raphaelmario.steiner@math.ethz.ch}}). Research supported by the SNSF Ambizione Grant No. 216071.}

\footnotetext[5]{Mathematical Institute, University of Oxford, UK 
(\textsf{\href{mailto:youri.tamitegama@keble.ox.ac.uk}{youri.tamitegama@keble.ox.ac.uk}}).}

\footnotetext[6]{All Souls College and Mathematical Institute, University of Oxford, UK (\textsf{\href{mailto:jane.tan@all-souls.ox.ac.uk}{jane.tan@all-souls.ox.ac.uk}}).}
 
\footnotetext[7]{School of Mathematics, Monash University, Melbourne, Australia (\textsf{\{\href{mailto:junghon.yip@monash.edu}{junghon.yip},\href{mailto:david.wood@monash.edu}{david.wood}\}\allowbreak@monash\allowbreak.edu}). Research of Wood is supported by the Australian Research Council and by NSERC. }

\renewcommand{\thefootnote}{\arabic{footnote}}
\section{Introduction}

The 4-colour theorem states that every planar graph is 4-colourable~\citep{AH89,RSST97}. By the Kuratowski--Wagner theorem~\citep{Kuratowski30,Wagner37}, the 4-colour theorem is equivalent to saying that every $K_5$-minor-free and $K_{3,3}$-minor-free graph is 4-colourable. In fact, using this and Wagner's characterisation of $K_5$-minor-free graphs, it is straightforward to show that every $K_5$-minor-free graph is 4-colourable.

We prove a substantial generalisation of the latter result via the following definition introduced by \citet{IW25}. A \defn{dominating $K_t$-model} in a graph $G$ is a sequence $(T_1,\dots,T_t)$ of pairwise disjoint connected subgraphs of $G$, such that whenever $1\leq i<j\leq t$ every vertex in $T_j$ has a neighbour in $T_i$. That is, each $T_i$ dominates $T_{i+1} \cup \dots \cup T_t$. Replacing ``every vertex in $T_j$'' by ``some vertex in $T_j$'' retrieves the standard definition of $K_t$-model, which is equivalent to a $K_t$-minor in $G$. Our main result is the following.

\begin{thm}
\label{thm:d4ct}
Every graph with no dominating $K_5$-model is $4$-colourable. 
\end{thm}

The proof of \cref{thm:d4ct} (which employs the 4-colour theorem for planar graphs) is presented in \cref{Proof}. Before that we make several comments about \cref{thm:d4ct}. 

First, note that \cref{thm:d4ct} is equivalent to saying that every 5-chromatic  graph has a dominating $K_5$-model $(T_1, \dots, T_5)$. If $v$ is any vertex in $T_5$ and $w$ is any neighbour of $v$ in $T_4$, then $(T_1,T_2,T_3,\set{v},\set{w})$ is a dominating $K_5$-model \citep{IW25}. Thus there is a dominating $K_5$-model with two singleton branch-sets. This strengthens a result of \citet[Cor.~1.4]{MS24} who showed that every 5-chromatic graph contains a (not necessarily dominating) $K_5$-model with at least one singleton branch-set. Two singleton branch-sets is optimal here, since there are 5-chromatic graphs with arbitrarily large girth~\citep{Erdos59}, and no complete graph model in any such graph has three branch-sets of bounded size. One final remark on the structure of the branch-sets: we may further take $T_3$ to be a path and $G[V(T_3) \cup V(T_4) \cup V(T_5)]$ to be an induced cycle, by reducing $T_3$ to a shortest path from a neighbour of $v$ to a neighbour of $w$.

The class of graphs with no dominating $K_5$-model is significantly more general than the class of graphs with no $K_5$-minor. For example, the 1-subdivision of $K_n$ has no dominating $K_4$-model~\citep{IW25}, but obviously contains $K_n$ as a minor. Moreover, no graph with maximum degree $3$, including random cubic graphs and cubic expanders, has a dominating $K_5$-model. In fact, every graph, in which no two vertices with degree at least $4$ are adjacent, has no dominating $K_5$-model (since if $(T_1,T_2,T_3,\set{v},\set{w})$ is a dominating $K_5$-model, then $v$ and $w$ are adjacent vertices both with degree at least $4$). There are even $4$-connected graphs that contain a $K_5$-minor but contain no dominating $K_5$-model. For example, let $G$ be obtained from the complete bipartite graph $K_{5,5}$ by removing the edges of a perfect matching. Note that $G$ has a perfect matching whose contraction gives a $K_5$-minor. However, it does not contain a dominating $K_5$-model\footnote{Let $\set{A, B}$ be the bipartition of the $K_{5, 5}$ and $M$ be the matching removed. Suppose towards a contradiction that $G$ has a dominating $K_5$-model $(T_1,\dots,T_5)$. From the above discussion, we may assume that $T_4 = \set{p}$ and $T_5 = \set{q}$ where $p \in A$ and $q \in B$. Since $G$ is triangle-free, each of $T_1,T_2,T_3$ has at least two vertices. Suppose that $\abs{T_i} = \abs{T_j} = 2$ for distinct $i,j\in\set{1, 2, 3}$. Let $T_i=ab$ and $T_j=cd$ with $a,c\in A$ and $b,d\in B$. Then $\set{a,b,p},\set{c,d,q}$ is a bipartition of $K_{3,3}$, which does not exist in $G$. So at most one of $T_1,T_2,T_3$ has exactly two vertices. Since $\abs{V(G)} = 10$, two of $T_1,T_2,T_3$ have exactly three vertices, and the other has exactly two vertices. In particular, every vertex of $G$ is used by $T_1, \dots, T_5$ and $T_1$ is a path with two or three vertices. Let $v \in V(T_1)$ be the only vertex of $T_1$ in its side of the bipartition. Then the vertex matched with $v$ in $M$ is not in $T_1$ and not dominated by $T_1$, which is the required contradiction.}.

Finally, we place  \cref{thm:d4ct} in the context of some more general conjectures. \citet{Hadwiger43} famously conjectured that every graph with no $K_t$-minor is $(t-1)$-colourable. This conjecture is straightforward for $t\leq 4$ \citep{Hadwiger43,Dirac52}. It is true for $t=5$ as mentioned above, it is true for $t=6$ as proved in \citet{RST-Comb93}, and it is open for $t\geq 7$. The dominating Hadwiger conjecture, stated by \citet{IW25}, asserts that every graph with no dominating $K_t$-model is $(t-1)$-colourable. This conjecture implies and is significantly stronger than Hadwiger's conjecture. \citet{IW25} proved the dominating Hadwiger conjecture for $t\leq 4$. \cref{thm:d4ct} proves it for $t=5$.

Haj\'{o}s' conjecture %\footnote{The precise origin of the conjecture is unclear; see the introduction of \citet{Thomassen05}.}
is a famous strengthening of Hadwiger's conjecture stating that, for every $t\in \mathbb{N}$, every graph $G$ with $\chi(G)\ge t$ contains a \emph{subdivision of $K_t$} as a subgraph. For $t\le 4$, this is equivalent to Hadwiger's conjecture, since for $t\le 4$ a graph contains $K_t$ as a minor if and only if it contains a subdivision of $K_t$. However, starting from $t=5$, Haj\'{o}s' conjecture becomes significantly stronger than Hadwiger's conjecture, since the equivalence between minor and subdivision containment breaks for $t\ge 5$. \citet{Catlin79} famously showed that Haj\'{o}s' conjecture is false for every $t\ge 7$; also see~\citet{Thomassen05}. Indeed, while Hadwiger's conjecture and the dominating Hadwiger conjecture are true for almost every graph \citep{BCE80,IW25}, 
Haj\'{o}s' conjecture is false for almost every graph \citep{EF81}. Haj\'{o}s' conjecture remains open in the cases $t\in \set{5,6}$. In particular, the case $t=5$ has attracted a significant amount of attention~\cite{MR2232386,MR3436970,MR4371436,MR4429165} and is closely connected to the now-resolved longstanding Kelmans-Seymour conjecture~\cite{kelmans1,kelmans2,kelmans3,kelmans4}. 

To demonstrate the significant difference between Hadwiger's conjecture and Haj\'{o}s' conjecture in the case $t=5$, note that there exists an explicit family of $22$ non-isomorphic graphs\footnote{Including the Petersen graph, $H_{22}$.} $H_1, \dots,H_{22}$, depicted in \cref{fig:K5splits}, such that a graph $G$ contains $K_5$ as a minor if and only if it contains a subdivision of at least one of $H_1,\dots,H_{22}$. While Hadwiger's conjecture for $t=5$ only guarantees a subdivision of one of these $22$ graphs in any $5$-chromatic graph, Haj\'{o}s' conjecture makes the much stronger claim that, out of these $22$ graphs, only one is necessary, namely $K_5$ itself. 

\begin{figure}[ht]
    \centering
    {$H_1$\hspace{-2ex}}
    \includegraphics{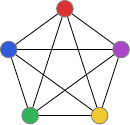}
    %\qquad 
    {\hspace*{1ex}$H_2$\hspace{-2ex}}
    \includegraphics{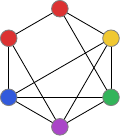}
%    \qquad 
     {\hspace*{1ex}$H_3$\hspace{-2ex}}
   \includegraphics{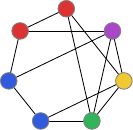}
%    \qquad 
     {\hspace*{1ex}$H_4$\hspace{-2ex}}
    \includegraphics{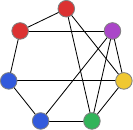}
%    \qquad 
     {\hspace*{1ex}$H_5$\hspace{-2ex}}
    \includegraphics{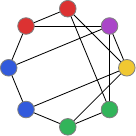}
     \\[2em]
%    \qquad 
     {$H_6$\hspace{-2ex}}
    \includegraphics{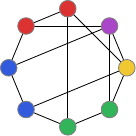}
%    \qquad 
     {\hspace*{1ex}$H_7$\hspace{-2ex}}
    \includegraphics{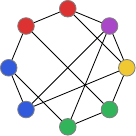}
%    \qquad 
     {\hspace*{1ex}$H_8$\hspace{-2ex}}
    \includegraphics{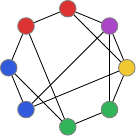}
%    \qquad 
     {\hspace*{1ex}$H_9$\hspace{-2ex}}
    \includegraphics{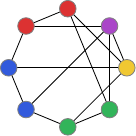}
%    \qquad 
     {\hspace*{1ex}$H_{10}$\hspace{-2ex}}
    \includegraphics{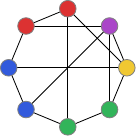}
     \\[2em]
%    \qquad 
     {$H_{11}$\hspace{-2ex}}
    \includegraphics{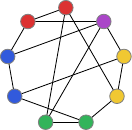}
%    \qquad 
     {\hspace*{1ex}$H_{12}$\hspace{-2ex}}
    \includegraphics{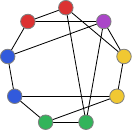}
%    \qquad 
     {\hspace*{1ex}$H_{13}$\hspace{-2ex}}
    \includegraphics{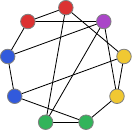}
%    \qquad 
     {\hspace*{1ex}$H_{14}$\hspace{-2ex}}
    \includegraphics{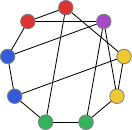}
%    \qquad 
     {\hspace*{1ex}$H_{15}$\hspace{-2ex}}
    \includegraphics{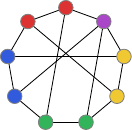}
     \\[2em]
%    \qquad 
     {\hspace*{1ex}$H_{16}$\hspace{-2ex}}
    \includegraphics{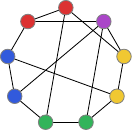}
%    \qquad 
     {\hspace*{1ex}$H_{17}$\hspace{-2ex}}
    \includegraphics{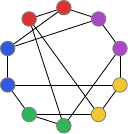}
%    \qquad 
     {\hspace*{1ex}$H_{18}$\hspace{-2ex}}
    \includegraphics{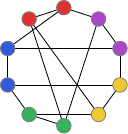}
%    \qquad 
     {\hspace*{1ex}$H_{19}$\hspace{-2ex}}
    \includegraphics{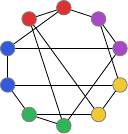}
%    \qquad 
     {\hspace*{1ex}$H_{20}$\hspace{-2ex}}
    \includegraphics{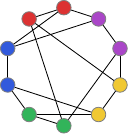}
     \\[2em]
%    \qquad 
     {\hspace*{1ex}$H_{21}$\hspace{-2ex}}
    \includegraphics{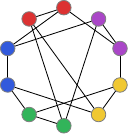}
%    \qquad 
     {\hspace*{1ex}$H_{22}$\hspace{-2ex}}
    \includegraphics{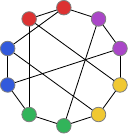}
    \caption{The 22 graphs obtained by splitting $K_5$, where vertices of the same colour represent pairs split from the same vertex of $K_5$. }
    \label{fig:K5splits}
\end{figure}

\Cref{thm:d4ct} quickly implies the following result, which shows that the list of graphs necessary in this statement can be reduced from the list of $22$ graphs given by Hadwiger's conjecture to only two graphs: $K_5=H_1$ and the unique $6$-vertex graph $\widehat{K}_5=H_2$ obtained from $K_5$ by splitting one vertex into two adjacent degree-$3$ vertices. This represents significant new evidence for the correctness of Haj\'{o}s' conjecture for $t=5$.

\begin{cor}
\label{NearHajos}
    Every graph $G$ with $\chi(G)\ge 5$ contains a subdivision of $K_5$ or $\widehat{K}_5$.
\end{cor}

\begin{proof}
    By \cref{thm:d4ct}, there exists a dominating $K_5$-model $(T_1,\dots,T_5)$ in $G$. Clearly, the subgraph of $G$ induced by the vertices in $V(T_2)\cup\dots\cup V(T_5)$ contains a dominating $K_4$-model and hence contains $K_4$ as a minor. In particular, $G[V(T_2)\cup\dots\cup V(T_5)]$ contains a subdivision $S$ of $K_4$ as a subgraph. Let $b_1,\dots,b_4$ be the four vertices of degree $3$ in $S$ (the ``branch-vertices'' of the subdivision). Then by definition of a dominating model, we have that for each $i\in \set{1,2,3,4}$ there exists some $v_i\in V(T_1)$ such that $v_ib_i\in E(G)$. Now consider the smallest subtree $T$ of $T_1$ containing all of $\set{v_1,\dots,v_4}$. It is easy to see that every leaf of $T$ must lie in $\set{v_1,\dots,v_4}$. It follows that $T$ can be written as the union of five internally vertex-disjoint paths $P, P_1,\dots,P_4$ such that $P_i$ has one endpoint in $v_{\pi(i)}$ for each $i\in\set{1,2,3,4}$ and some permutation $\pi\in S_4$, and such that $P_1,P_2$ start at one of the endpoints of $P$ while $P_3,P_4$ start at the other endpoint of $P$. It can now be easily seen that the graph $S\cup P\cup P_1\cup P_2\cup P_3\cup P_4 \cup \set{v_1b_1, v_2b_2,v_3b_3,v_4b_4}$ forms a subdivision of $K_5$ if $P$ has length $0$, and a subdivision of $\widehat{K}_5$ if $P$ has positive length. Hence, $G$ indeed contains a subdivision of $K_5$ or $\widehat{K}_5$, concluding the proof. 
\end{proof}

Note that \cref{NearHajos} can be strengthened further by requiring that two incident edges of $K_5$ or two incident edges of $\widehat{K_5}$ in its unique $K_4$-subgraph are not subdivided (that is, remain edges) in the subdivision found in $G$. This is because every graph with a dominating $K_4$-model contains a subdivision of $K_4$ with two incident edges not subdivided and so, in the proof of \cref{NearHajos}, one may assume that $S$ has this property.

\section{The Proof} 
\label{Proof}

In this paper all graphs are simple. We frequently consider a minor of a graph that is obtained by contracting a set of edges. After the contraction, all resulting loops are removed and parallel edges are replaced by single edges.

We prove \cref{thm:d4ct} by induction with a strengthened induction hypothesis. For a dominating $K_5$-model $\TT = (T_1,\dots,T_5)$ in a graph $G$ and a vertex $v\in V(G)$, define
\[
\mathdefn{\ind_\TT(v)} \coloneqq \begin{cases}
    i & \text{if $v \in V(T_i)$ for $1\le i\le 5$}; \\
    0 & \text{if $v \in V(G)\setminus V(T_1\cup\dots\cup T_5)$}.
\end{cases}
\]
Consider an ordered clique $L = (v_1,\dots,v_k)$ in a graph $G$ with $k\leq 2$. A dominating $K_5$-model $\TT$ is \defn{$L$-compatible} if $\ind_\TT(v_i)\leq i$ for each $i\in\set{1,\dots,k}$, and if $k=2$ and $\ind_\TT(v_2) = 2$, then $\ind_\TT(v_1) = 1$. 

\begin{lem}[Main Induction Hypothesis]\label{lem:induction-hypothesis}
    For every graph $G$ and every ordered clique $L=(v_1,\dots,v_k)$ in $G$ with $k\leq 2$, either:
    \begin{itemize}
        \item $G$ is 4-colourable, or 
        \item $G$ has an $L$-compatible dominating $K_5$-model. 
    \end{itemize}
\end{lem}

We introduce some tools before proving \cref{lem:induction-hypothesis}. The following lemma is used to carry $L$-compatible dominating $K_5$-models through the induction steps. If edge contractions are done, the position of the resulting vertex in $L$ is the lowest position of any vertex of $L$ among the identified vertices (and if no vertex of $L$ is among the contracted vertices, then the resulting vertex is not in $L$).

\begin{lem}\label{lem:contraction}
    Let $L=(v_1,\dots,v_k)$ be an ordered clique in a graph $G$ with $k\leq 2$. 
    \begin{enumerate}[label=\textup{(\alph{*})}]
        \item\label{part:contraction-1} Let $H_1$ be a connected subgraph of $G$ that contains $v_1$. Let $(G',L')$ be the pair obtained from $(G,L)$ after contracting the edges of $H_1$. If $G'$ has an $L'$-compatible dominating $K_5$-model, then $G$ has an $L$-compatible dominating $K_5$-model. 
        \item\label{part:contraction-2} Let $H_2$ be a connected subgraph of $G[N_G(v_1)]$ that contains $v_2$. 
        Let $(G',L')$ be the pair obtained from $(G,L)$ after contracting the edges of $H_2$.
        If $G'$ has an $L'$-compatible dominating $K_5$-model, then $G$ has an $L$-compatible dominating $K_5$-model. 
        \item\label{part:deletion} Let $A \subseteq V(G)$. If $G[A]$ has an $(L \cap A)$-compatible dominating $K_5$-model, then $G$ has an $L$-compatible dominating $K_5$-model.
        \item\label{part:extension} Let $L'$ be an ordered clique on at most two vertices such that $L$ is an initial segment of $L'$. If $G$ has an $L'$-compatible dominating $K_5$-model, then $G$ has an $L$-compatible dominating $K_5$-model.
    \end{enumerate}
\end{lem}

\begin{proof}
    We prove the items in turn. Let $\TT = (T_1, \dots, T_5)$ be an $L'$-compatible dominating $K_5$-model in $G' = G/E(H_1)$.
    Let $v_1^\ast$ be the vertex obtained from the contracted edges. Since $v_1 \in V(H_1)$, $v_1^\ast$ is the first vertex of $L'$ and so $\ind_\TT(v_1^\ast) \leq 1$.
    First, suppose that $\ind_\TT(v_1^\ast) = 0$. If $v_2 \in V(H_1)$, then $L$ is disjoint from $\bigcup_{i = 1}^5 T_i$ and so $\TT$ is $L$-compatible in $G$. If $v_2 \notin V(H_1)$, then, since $\ind_\TT(v_1^\ast) = 0$ and $\TT$ is $L = (v_1^\ast, v_2)$-compatible, $\ind_\TT(v_2) \leq 1$. 
    Again, $\TT$ is $L$-compatible in $G$.
    If $\ind_\TT(v_1^\ast) = 1$, then $(T'_1, T_2, \dots, T_5)$ is an $L$-compatible dominating $K_5$-model in $G$, where $T'_1$ is obtained from $T_1$ by uncontracting the edges of $H_1$ regardless of whether $v_2 \in V(H_1)$.

    Let $\TT = (T_1, \dots, T_5)$ be an $L'$-compatible dominating $K_5$-model in $G' = G/E(H_2)$. Let $v_2^\ast$ be the vertex obtained from the contracted edges. Since $v_2 \in V(H_2)$ and $v_1 \notin V(H_2)$, $v_2^\ast$ is the second vertex of $L'$ and so $\ind_\TT(v_2^\ast) \leq 2$. If $\ind_\TT(v_2^\ast) = 0$, then $\TT$ is also $L$-compatible in $G$. If $\ind_\TT(v_2^\ast) = 1$, then $(T_1', T_2, \dots, T_5)$ is an $L$-compatible dominating $K_5$-model in $G$, where $T_1'$ is obtained from $T_1$ by uncontracting the edges of $H_2$. Finally, if $\ind_\TT(v_2^\ast) = 2$, then $\ind_\TT(v_1) = 1$ and so $(T_1, T'_2, T_3, T_4, T_5)$ is $L$-compatible where $T'_2$ is obtained from $T_2$ by uncontracting the edges of $H_2$. Also, since $V(H_2) \subseteq N_G(v_1)$, it remains a dominating $K_5$-model.

    Let $\TT = (T_1, \dots, T_5)$ be an $(L \cap A)$-compatible dominating $K_5$-model in $G[A]$. If $v_1 \in A$, then $\TT$ is an $L$-compatible dominating $K_5$-model in $G$. If $v_1 \notin A$, then $\ind_\TT(v_2) \leq 1$ (since $v_2$ is either the first vertex of $L \cap A$ or is not in $G[A]$) and so $\TT$ is still $L$-compatible in $G$.

    Let $\TT = (T_1, \dots, T_5)$ be an $L'$-compatible dominating $K_5$-model in $G$. Since $L$ is an initial segment of $L'$, $\TT$ is also $L$-compatible.
\end{proof}

Recall that a \defn{proper separation} $(A, B)$ of a graph $G$ is a pair $A, B \subset V(G)$ where $A \cup B = V(G)$, both $A \setminus B$ and $B \setminus A$ are non-empty, and there is no edge from $A \setminus B$ to $B \setminus A$. Further, $(A, B)$ is a \defn{proper $(\leq t)$-separation} (\defn{proper $t$-separation}) if it is a proper separation and $\abs{A \cap B} \leq t$ ($\abs{A \cap B} = t$). Recall that if a graph is not $(k + 1)$-connected, then either it is $K_{k + 1}$ or it has a proper $(\leq k)$-separation.

Here is a simple lemma that stitches together colourings of $G[A]$ and $G[B]$.

\begin{lem}\label{lem:colouring}
    Let $c \geq 1$. Suppose that $(A, B)$ is a proper $(\leq c)$-separation of a graph $G$. Let $G'$ be obtained from $G[B]$ by adding a clique $K$ on $A \cap B$. Suppose that $G[A]$ is $c$-colourable and, for every $E' \subseteq E(K)$, $G'/E'$\footnote{\defn{$G'/E'$} is the graph obtained from $G'$ by contracting the edges of $E'$.} is $c$-colourable. Then $G$ is $c$-colourable.
\end{lem}

\begin{proof}
    Let $\chi_A \colon A \to \set{1, \dots, c}$ be a $c$-colouring of $G[A]$ and define $S_i \coloneqq \chi_A^{-1}(i) \cap A \cap B$ for $i \in \set{1, \dots, c}$. Let $G''$ be obtained from $G'$ by contracting each $S_i$ into a single vertex $s_i$ (if $S_i = \emptyset$, then there is no $s_i$). By supposition, $G''$ is $c$-colourable. Note that the $s_i$ form a clique in $G''$.
    Let $\chi_B \colon V(G'') \to \set{1, \dots, c}$ be a $c$-colouring of $G''$ with $\chi_B(s_i) = i$ if $s_i$ exists. Let $\chi$ be $\chi_A$ on $A$ and $\chi_B$ on $B \setminus A$. Then $\chi$ is a $c$-colouring of $G$.
\end{proof}

We need the following technical strengthening of a result by \citet{IW25} that finds dominating $K_4$-models. 

\begin{lem}\label{lem:dominating-K4}
    Let $G$ be a connected graph with at least four vertices. Let $L$ be a clique in $G$ on at most two vertices, such that every vertex of degree at most 2 in $G$ is in $L$, and if some vertex in $L$ has degree at least 3 in $G$, then $\abs{L} = 1$. Then $G$ has a dominating $K_4$-model $(T_1, T_2, T_3, T_4)$ with $L \subseteq T_1$.
\end{lem}

\begin{proof}
    We proceed by induction on $\abs{V(G)}$. If $\abs{V(G)} = 4$, then $G = K_4$ and $\abs{L} \leq 1$, and the claim is immediate. Now assume $\abs{V(G)} \geq 5$ and the result holds for graphs with less than $\abs{V(G)}$ vertices.
    
    Suppose $G$ is not 2-connected. So $G$ has a proper 1-separation $(A, B)$. Let $v$ be the unique vertex in $A \cap B$. Since $L$ is a clique, without loss of generality, $L \subseteq A$. 
    Note every vertex of $B \setminus A$ has degree at least $3$ and so $\abs{B} \geq 4$.
    Apply induction to $G[B]$ with $L' \coloneqq \set{v}$, which
    is valid since $v$ is the only possible vertex with degree at most two in $G[B]$: $G[B]$ has a dominating $K_4$-model $(T_1, T_2, T_3, T_4)$ with $v \in T_1$. But then
    $(T_1 \cup G[A],T_2,T_3,T_4)$ is a dominating $K_4$-model in $G$ with $L \subseteq T_1 \cup G[A]$. Now assume that $G$ is 2-connected.
    
    Suppose that $G - L$ is a forest. If $G - L$ has an isolated vertex $v$, then the only neighbours of $v$ are in $L$, implying $\deg(v) \leq \abs{L} \leq 2$ and $v$ is in $L$, which is a contradiction. Thus $G - L$ contains at least two leaves, denoted $v$ and $w$. Since $v$ and $w$ are not in $L$, they both have degree at least  $3$, implying that $v$ and $w$ both have at least two neighbours in $L$. So $\abs{L} = 2$, and every vertex in $L$ is adjacent to both $v$ and $w$. Since $L$ is a clique, every vertex in $L$ has degree at least $3$, which contradicts the assumption that $\abs{L} = 1$ in this case. Thus, there is a pair $(H, C)$ such that $C$ is a cycle, $H$ is a component of $G - C$, and $L \subseteq H$. Choose such a pair with $\abs{V(H)}$ maximal and, subject to this, $\abs{V(C)}$ minimal. Then $C$ is an induced cycle. Since $G$ is 2-connected, there are distinct vertices $y$ and $z$ in $C$ both adjacent to $H$.
    
    Suppose there is a vertex $v \in V(C)$ not adjacent to $H$. So $v \notin \set{y, z}$. 
    Since $v \notin L$, $\deg(v) \geq 3$. Since $C$ is induced, $v$ is adjacent to a component $J$ of $G-C$ distinct from $H$. 
    Since $G$ is 2-connected, there is a vertex $w$ in $C-v$ adjacent to $J$. Without loss of generality, $w\neq y$. Let $P$ be the $vw$-path in $C$ avoiding $y$. 
    Let $Q$ be a $vw$-path through $J$. So $P \cup Q$ is a cycle, and there is a component of $G - (P\cup Q)$ that includes $H$ and $y$ (and so contains $L$), which contradicts the maximality of $\abs{V(H)}$.
    
    Now assume that every vertex in $C$ is adjacent to $H$. Let $xy$ be an edge of $C$. Then $(H,C-x-y,\set{x},\set{y})$ is a dominating $K_4$-model with $L \subseteq H$.
\end{proof}

Note that \cref{lem:dominating-K4} implies that every graph on at least four vertices that does not contain a dominating $K_4$-model has two non-adjacent vertices of degree at most 2. The corresponding result for $K_4$-minor-free graphs is well-known (see, for example, \citep[Prop.~7.3.1]{Diestel5}).

Our final preliminary lemma deals with one of the cases where $G$ is 4-connected.

\begin{lem}\label{lem:4-connected-triangle}
    Let $G$ be a 4-connected graph and $e = v_1 v_2$ an edge of $G$. If $G$ contains a triangle that is vertex-disjoint from $e$, then either $G$ is planar or $G$ has a $(v_1, v_2)$-compatible dominating $K_5$-model.
\end{lem}

\begin{proof}
The proof uses the following concept introduced by 
\citet*{RST-Comb93}. A \defn{triad with feet $x, y, z$} consists of an $ax$-path, an $ay$-path, and an $az$-path such that the only vertex on at least two of the paths is $a$, where $a \notin \set{x, y, z}$ (in particular, each path has length at least $1$ and the degrees of $x, y, z$ in the triad are all $1$). 
    
    Let $v_3 v_4 v_5$ be a triangle that is vertex-disjoint from $e$.
    Since $G$ is 4-connected, $G$ has no proper $(\leq 3)$-separation.
    Thus (3.4) and (3.5) of \citet*{RST-Comb93} imply that either $G$ is planar or there are two edge-disjoint triads $T_1, T_2$, both with feet $v_3, v_4, v_5$, such that $V(T_1) \cap V(T_2) = \set{v_3, v_4, v_5}$ (this is a `legless tripod on $v_3, v_4, v_5$' in the language of \cite{RST-Comb93}).
    Since $G$ is 4-connected, $G - \set{v_3, v_4, v_5}$ is connected and so without loss of generality there is a path from $v_1$ to $T_1$ that is vertex-disjoint from $T_2$ (this includes the case where $v_1 \in V(T_1)$). In particular, there is a component $W$ of $G - V(T_2)$ such that $\set{v_1} \cup (T_1 - \set{v_3, v_4, v_5}) \subseteq W$. 

    A triad $T$ is \defn{lean} if, for every triad $T'$ with the same feet as $T$, $V(T')$ is not a proper subset of $V(T)$. Since $G$ has no proper $(\leq 3)$-separation, (3.1) of \cite{RST-Comb93} implies that there is a lean triad $T$ with feet $v_3, v_4, v_5$ such that $G - V(T)$ is connected and contains $W$. We claim that $\TT = (G - V(T), T - \set{v_3, v_4, v_5}, v_3, v_4, v_5)$ is a $(v_1, v_2)$-compatible dominating $K_5$-model. 
    
    First $v_1 \in W \subseteq G - V(T)$ and so $\ind_\TT(v_1) = 1$. Since $v_3 v_4 v_5$ is disjoint from $e$, $\ind_\TT(v_2) \leq 2$ and so $\TT$ is $(v_1, v_2)$-compatible. Second, every part of $\TT$ is connected. Third, $v_3 v_4 v_5$ is triangle and so the dominating condition between the third, fourth, and fifth parts is satisfied. Next, fix $i \in \set{3, 4, 5}$. Vertex $v_i$ has a neighbour in each of $T - \set{v_3, v_4, v_5}$ and $T_1 - \set{v_3, v_4, v_5}$. Since $T_1 - \set{v_3, v_4, v_5} \subseteq W \subseteq G - V(T)$ the dominating condition holds between the third/fourth/fifth parts and the first/second parts of $\TT$. It remains to check that every vertex of $T - \set{v_3, v_4, v_5}$ has a neighbour outside $T$. Since $G$ is 4-connected, $G$ has minimum degree at least $4$, so it suffices to show that every vertex of $T - \set{v_3, v_4, v_5}$ has at most three neighbours in $V(T)$.

    \begin{figure}[htp]
        \centering
        \includegraphics[width=\textwidth]{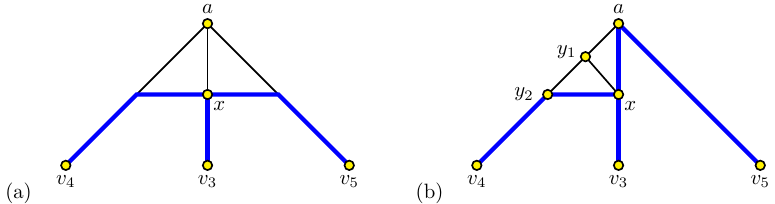}
        \caption{}\label{fig:triads}
    \end{figure}

    Suppose, towards a contradiction, that there is a vertex $x$ of $T - \set{v_3, v_4, v_5}$ with at least four neighbours in $T$.    Let $a$ be the common vertex of the three paths that make up $T$ and let $P_i$ be the $av_i$-path of $T$ (for $i = 3, 4, 5$). Since $T$ is lean, each $P_i$ is induced in $G$. In particular, $a$ only has three neighbours in $T$ and so $x \neq a$. Thus $x$ is an internal vertex of some $P_i$ which we may and will assume is $P_3$. Since $P_3$ is induced, $x$ has only two neighbours on $P_3$. If $x$ has a neighbour in each of $P_4 - a$ and $P_5 - a$, then there is a triad $T'$ with feet $v_3, v_4, v_5$ where $V(T') \subseteq V(T) \setminus \set{a}$ (see \cref{fig:triads}(a) where $x$ is the common vertex of the three paths). This contradicts the leanness of $T$. Thus, we may and will assume that $x$ has at least two neighbours on $P_4 - a$. In particular, there are vertices $y_1, y_2$ on $P_4 - a$ both adjacent to $x$ such that $y_1$ lies between $a$ and $y_2$ on $P_4$. But then there is a triad $T'$ with feet $v_3, v_4, v_5$ where $V(T') \subseteq V(T) \setminus \set{y_1}$ (see \cref{fig:triads}(b) where $x$ is the common vertex of the three paths). This contradicts the leanness of $T$, as required.
\end{proof}

We are now ready to prove \cref{lem:induction-hypothesis} which implies \cref{thm:d4ct}.

\begin{proof}[Proof of \cref{lem:induction-hypothesis}]
    Suppose, for a contradiction, that there is a pair $(G, L)$ where $G$ is a non-4-colourable graph, $L$ is an ordered clique in $G$ of size at most 2, and $G$ has no $L$-compatible dominating $K_5$-model. 
    Choose such a pair $(G, L)$ where $\abs{V(G)}$ is minimal and, subject to this, $\abs{L}$ is maximal. Since $G$ is not 4-colourable, $\abs{V(G)} \geq 5$.

    \begin{claim}
        Every vertex of $G$ has degree at least 4.
    \end{claim}

    \begin{proofclaim}
        Suppose that there is a vertex $x \in V(G)$ with $\deg(x) \leq 3$. By minimality of $\abs{V(G)}$, either $G - x$ is 4-colourable or $G - x$ has an $L \setminus \set{x}$-compatible dominating $K_5$-model. In the former case, $G$ is 4-colourable and in the latter case, \cref{lem:contraction}\ref{part:deletion} implies that $G$ has an $L$-compatible dominating $K_5$-model, which is the required contradiction.
    \end{proofclaim}
    
    \begin{claim}
        $G$ is 2-connected.
    \end{claim}

    \begin{proofclaim}
        If not, then, since $\abs{V(G)} \geq 5$, there is a proper $(\leq 1)$-separation $(A, B)$ of $G$. Since $G$ does not have an $L$-compatible dominating $K_5$-model, \cref{lem:contraction}\ref{part:deletion} implies that $G[A]$ does not have an $(L \cap A)$-compatible dominating $K_5$-model. Minimality of $\abs{V(G)}$ implies that $G[A]$ is 4-colourable. Similarly $G[B]$ is 4-colourable. Since $\abs{A \cap B} \leq 1$, \cref{lem:colouring} implies that $G$ is also 4-colourable, which is the required contradiction.
    \end{proofclaim}

    By maximality of $\abs{L}$ and \cref{lem:contraction}\ref{part:extension} we have $\abs{L} = 2$. Let $v_1, v_2$ be the vertices of $L$ (in that order).

    \begin{claim}
        $G$ is 3-connected. 
    \end{claim}

    \begin{proofclaim}
        If not, then, since $\abs{V(G)} \geq 5$, there is a proper 2-separation $(A, B)$ of $G$. Since $L$ is a clique, we may and will assume that $L \subseteq A$. Let $\set{x, y} = A \cap B$. 
        Since $G$ does not have an $L$-compatible dominating $K_5$-model, \cref{lem:contraction}\ref{part:deletion} implies that $G[A]$ does not, and so minimality of $\abs{V(G)}$ implies that $G[A]$ is 4-colourable.

        Without loss of generality, $v_1 \neq y$ and $v_2 \neq x$. Since $G$ is 2-connected, $y$ does not separate $v_1$ from $B \setminus A$ and so there is a (possibly empty) path $P_x$ from $v_1$ to $x$ in $G[A]$ that avoids $y$. Similarly there is a path $P_y$ from $v_2$ to $y$ in $G[A]$. 
        The walk $v_1 v_2 \cup P_y$ shows that there is a path from $v_1$ to $y$ in $G[A]$. Thus there is a connected subgraph $H_1$ of $G[A]$ that contains $v_1$ and $x$, does not contain $y$, but is adjacent to $y$. Contract $E(H_1)$ to get a vertex $v_1^\ast$ which is adjacent to $y$. Since $x$ was in $H_1$, we view $v_1^\ast$ as being in $A \cap B$.

        The resulting graph, when induced on $B$, is isomorphic to $G' \coloneqq G[B] + xy$.
        Let $L'$ be what $L$ has now become intersected with $B$. Since $v_1 \in V(H_1)$, the first vertex of $L'$ is $v_1^\ast$.
        Further, in $G'$, contract the edge $v_1^\ast y$. Let $L''$ be what $L'$ has now become. The resulting graph, $G''$, is isomorphic to $(G[B] + xy)/xy$. 

        Since $G$ has no $L$-compatible dominating $K_5$-model and the contracted subgraphs were connected and contained $v_1$, \cref{lem:contraction}\ref{part:contraction-1} implies that $G'$ does not have an $L'$-compatible dominating $K_5$-model and $G''$ does not have an $L''$-compatible $K_5$-dominating model. The minimality of $\abs{V(G)}$ implies that $G'$ and $G''$ are both 4-colourable.

        Recall that $G[A]$ is 4-colourable. Thus \cref{lem:colouring} implies that $G$ is 4-colourable, which is the required contradiction.
    \end{proofclaim}

    \begin{claim}
        $G$ is 4-connected.
    \end{claim}

    \begin{proofclaim}
        If not, then, since $\abs{V(G)} \geq 5$, there is a proper 3-separation $(A, B)$ of $G$ with $L \subseteq A$. Choose such a separation with $\abs{A}$ minimal.
        Let $\set{x, y, z} = A \cap B$. Since $G$ does not have an $L$-compatible dominating $K_5$-model neither does $G[A]$ and so, by minimality of $\abs{V(G)}$, $G[A]$ is 4-colourable.

        Since every vertex of $G$ has degree at least $4$, we have $\abs{A \setminus B} \geq 2$. Let $v \in A \setminus B$ be distinct from $v_1$. Duplicate the vertex $v$ to get a non-adjacent twin $v'$. Doing this does not change connectivity and so the resulting graph is 3-connected. By Menger's theorem, there are three vertex-disjoint paths between $\set{v_1, v, v'}$ and $\set{x, y, z}$. Without loss of generality, these are $P_x$ from $v_1$ to $x$, $P$ from $v$ to $y$, and $P'$ from $v'$ to $z$.
        All three paths must have all their vertices in $A\cup \set{v'}$, since $\set{x, y, z}$ separates $v_1, v, v'$ from $B \setminus A$. In particular, $P_x$ is a $v_1x$-path in $G[A]$, and in the original graph $G$, the paths $P$ and $P'$ correspond to paths in $G[A]$ from $v$ to $y$ and $v$ to $z$ that are vertex-disjoint from $P_x$. Thus there is a walk from $y$ to $z$ in $G[A]$ that is vertex-disjoint from $P_x$. In conclusion, $G[A]$ contains the $v_1x$-path $P_x$ and some $yz$-path that is vertex-disjoint from $P_x$.

        We now run a standard rerouting argument (see, for example, \cite[Lemma~2.1]{KNTWa}). Let $P_{yz}$ be a $yz$-path in $G[A]$ that is vertex-disjoint from $P_x$ and such that
        \begin{enumerate}[label=(\roman{*})]
            \item the number of vertices in the same component as $P_x$ in $G[A] - V(P_{yz})$ is maximised and, subject to this,
            \item $\abs{V(P_{yz})}$ is minimised.
        \end{enumerate}
        Replacing $P_{yz}$ by a strictly shorter path from $y$ to $z$ contained in $V(P_{yz})$ would not decrease the number of vertices counted in (i) and strictly decrease the number of vertices counted in (ii), a contradiction. Hence, $P_{yz}$ must be an induced path. Let $H$ be the component of $G[A] - V(P_{yz})$ that contains $P_x$. 
        We now show that $H$ is the only component of $G[A] - V(P_{yz})$.
        Let $y = a_0, a_1, \dots, a_r = z$ be distinct vertices of $P_{yz}$ appearing in that order on $P_{yz}$ such that $\set{a_1, \dots, a_{r - 1}}$ is exactly the set of internal vertices of $P_{yz}$ that have a neighbour in $H$. 
        
        Consider some component $C$ of $G[A] - V(P_{yz})$ that is not $H$ (there may be no such component).
        Further, suppose that $C$ has neighbours $b$ and $c$ on $P_{yz}$ such that there is some $a_i$ strictly between $b$ and $c$. Then rerouting $bP_{yz}c$ \footnote{We use \defn{$bP_{yz}c$} to denote the subpath of $P_{yz}$ with ends $b$ and $c$.} through $C$ increases the size of $H$ (it now contains $a_i$).
        This contradicts (i).
        Thus, for every component $C \neq H$ of $G[A] - V(P_{yz})$, there is some $i \in \set{0, \dots, r - 1}$ such that all neighbours of $C$ on $P_{yz}$ are in $a_iP_{yz}a_{i + 1}$. 
        
        For each $i \in \set{0, \dots, r - 1}$, let $\CC_i$ be the union of components of $G[A] - V(P_{yz})$ that are not equal to $H$ and whose neighbours on $P_{yz}$ are all in $a_iP_{yz}a_{i + 1}$. Every component of $G[A] - V(P_{yz})$, other than $H$, is contained in some $\CC_i$. Now, the neighbourhood, in $G[A]$, of a vertex of $\CC_i$ is contained in $\CC_i$ and $a_i P_{yz} a_{i + 1}$. This is also true for vertices of $P_{yz}$ strictly between $a_i$ and $a_{i + 1}$ (recall that $P_{yz}$ is induced and $\set{a_1, \dots, a_{r - 1}}$ is exactly the set of internal vertices of $P_{yz}$ that have a neighbour in $H$). Thus, in $G[A] - \set{a_i, a_{i + 1}}$, the component containing any vertex of $\CC_i$ is contained in the union of $\CC_i$ and the subpath of $P_{yz}$ strictly between $a_i$ and $a_{i + 1}$. Suppose some $\CC_i$ is non-empty. Then $\set{a_i, a_{i + 1}}$ separates $\CC_i$ from $\set{x, y, z}$ in $G[A]$ which implies that $\set{a_i, a_{i + 1}}$ separates $\CC_i$ from $\set{x, y, z}$ in $G$. This contradicts the 3-connectivity of $G$. Therefore, every $\CC_i$ is empty and so $H$ is indeed the only component of $G[A] - V(P_{yz})$.

        We now show that every vertex of $P_{yz}$ has a neighbour in $H$. If $a$ is an internal vertex of $P_{yz}$, then it has degree at least $4$ in $G$ and is not in $\set{x, y, z}$, so it has degree at least $4$ in $G[A]$. Since $P_{yz}$ is an induced path, $a$ has a neighbour in $G[A]$ that is not on $P_{yz}$ and so is in $H$. Next consider $y$ and suppose towards a contradiction that $y$ has no neighbours in $H$. Since $P_{yz}$ is induced, $y$ has only one neighbour in $P_{yz}$ and so only one neighbour in $G[A]$, which we denote by $y'$. Vertex $y$ is neither equal nor adjacent to $v_1 \in H$, so $y \notin L$.
        But now $(A \setminus \set{y}, B \cup \set{y'})$ is a proper 3-separation of $G$ (the intersection is $\set{x, y', z}$) with $L \subseteq A \setminus \set{y}$. This contradicts the minimality of $\abs{A}$. Therefore $y$ has a neighbour in $H$. Repeating the same argument for $z$, we find that also $z$, and hence indeed every vertex of the path $P_{yz}$, must have some neighbour in $H$.

        We now contract some edges of $G[A]$ to create a clique on $A \cap B$. First, in the graph $G$ contract all of $H$ into a single vertex $v_1^\ast$. Let $G'$ and $L'$ be what $G$ and $L$ have become. Since $G$ does not have an $L$-compatible dominating $K_5$-model and $v_1 \in V(H)$, \cref{lem:contraction}\ref{part:contraction-1} implies that $G'$ does not have an $L'$-compatible dominating $K_5$-model. Since $v_1 \in V(H)$, $v_1^\ast$ is the first vertex of $L'$. Now contract the edges of $P_{yz}$ until $y$ and $z$ are adjacent. Let $G''$ and $L''$ be what $G'$ and $L'$ have become. Since every vertex of $P_{yz}$ had a neighbour in $H$, we have $V(P_{yz}) \subseteq N_{G'}(v_1^\ast)$ and so \cref{lem:contraction}\ref{part:contraction-2} implies that $G''$ does not have an $L''$-compatible dominating $K_5$-model. Note that $G'' \cong G[B] + E(K)$ where $K$ is a clique on $\set{x, y, z}$.

        Let $E'' \subseteq E(K)$. Let $G''' \coloneqq G''/E''$ and $L'''$ be what $G''$ and $L''$ become after contracting the edges of $E''$. If $E'' = \emptyset$, then $G''' = G''$ does not have an $L''' = L''$-compatible dominating $K_5$-model and so is 4-colourable by minimality of $\abs{V(G)}$. 
        Otherwise, since $K$ is a triangle, the edges of $E''$ form a connected graph (which does not necessarily span $V(K)$). 
        If that subgraph contains $v_1^\ast$, then \cref{lem:contraction}\ref{part:contraction-1} implies that $G'''$ does not contain an $L'''$-compatible dominating $K_5$-model and so is 4-colourable by minimality of $\abs{V(G)}$. Otherwise $E'' = \set{yz}$. If neither $y$ nor $z$ is in $L''$ (this occurs if $v_2 \in H$), then $L'' = \set{v_1^\ast}$. Now add $y$ to $L''$: by \cref{lem:contraction}\ref{part:extension}, $G''$ still does not have an $L''$-compatible dominating $K_5$-model. Thus, we may and will assume that one of $y$ or $z$ is in $L''$. Then \cref{lem:contraction}\ref{part:contraction-2} implies that $G'''$ does not contain an $L'''$-compatible dominating $K_5$-model and so is 4-colourable by minimality of $\abs{V(G)}$.

        We have shown that, for every $E'' \subseteq E(K)$, $G'' \cong G[B] + E(K)$ satisfies $\chi(G''/E'') \leq 4$. Recall that $G[A]$ is 4-colourable. \Cref{lem:colouring} implies that $G$ is 4-colourable which is the required contradiction.
    \end{proofclaim}

    Now, if $G$ contains a triangle that is vertex-disjoint from $L$, then, by \cref{lem:4-connected-triangle}, either $G$ is planar (and so 4-colourable) or $G$ has an $L$-compatible dominating $K_5$-model, a contradiction. Thus $G - V(L)$ is triangle-free.

    Let $N = G[N_G(v_1)]$. Note that $v_2 \in V(N)$ and $N$ has at least four vertices so there is $x \in V(N) \setminus \set{v_2}$. Suppose that every vertex of $N - \set{v_2}$ has degree at least $3$ in $N$. Let $C$ be the component of $N$ containing $x$. Then $C$ is connected and has at least four vertices. Let $L' = V(C) \cap \set{v_2}$ so $L'$ is a clique on at most one vertex.  Also, every vertex of degree at most 2 in $C$ is in $L'$. \Cref{lem:dominating-K4} then implies that $C$ has a dominating $K_4$-model $(T_2, T_3, T_4, T_5)$ with $L' \subseteq T_2$. In particular, $v_2 \notin T_3 \cup T_4 \cup T_5$. But then $(\set{v_1}, T_2, T_3, T_4, T_5)$ is an $L$-compatible dominating $K_5$-model in $G$, a contradiction.

    Thus, there is some vertex $u \in V(N) \setminus \set{v_2}$ with at most two neighbours in $N$. Let $M$ be the set of neighbours of $u$ in $V(G) \setminus (\set{v_1} \cup N_G(v_1))$.  Since $G - V(L)$ is triangle-free, $M$ is an independent set. Since $M$ is disjoint from $N_G(v_1)$, $M \cup \set{v_1}$ is an independent set.
    Let $G'$ be obtained from $G$ by contracting the edges of $\set{v_1u} \cup \set{um \colon m \in M}$ (these span a connected subgraph of $G$) into a new vertex $v_1^\ast$. Let $L'$ be what $L$ has now become. By \cref{lem:contraction}\ref{part:contraction-1}, $G'$ does not have an $L'$-compatible dominating $K_5$-model and so, by minimality of $\abs{V(G)}$, the graph $G'$ has a 4-colouring $\chi$. We 4-colour $G$ as follows. Use colouring $\chi$ on $G - (\set{v_1, u} \cup M)$ and colour the vertices of $\set{v_1} \cup M$ with colour $\chi(v_1^\ast)$. It remains to colour $u$. Since $\abs{N_G(u) \setminus (\set{v_1} \cup M)} = \deg_N(u) \leq 2$, at most three colours appear on the neighbourhood of $u$ in $G$ and so there is a colour available for $u$. This gives a 4-colouring of $G$. This contradiction completes the proof.
\end{proof}

\subsection*{Acknowledgements}

Research partially completed at the workshop `\href{https://www.matrix-inst.org.au/events/global-structure-and-geometry-of-graphs/}{Global Structure and Geometry of Graphs}' at MATRIX in Creswick, Australia, April 2026.

{\fontsize{10pt}{11pt}\selectfont
%\bibliographystyle{DavidNatbibStyle}
%\bibliography{DavidBibliography}
\def\soft#1{\leavevmode\setbox0=\hbox{h}\dimen7=\ht0\advance \dimen7 by-1ex\relax\if t#1\relax\rlap{\raise.6\dimen7 \hbox{\kern.3ex\char'47}}#1\relax\else\if T#1\relax \rlap{\raise.5\dimen7\hbox{\kern1.3ex\char'47}}#1\relax \else\if d#1\relax\rlap{\raise.5\dimen7\hbox{\kern.9ex \char'47}}#1\relax\else\if D#1\relax\rlap{\raise.5\dimen7 \hbox{\kern1.4ex\char'47}}#1\relax\else\if l#1\relax \rlap{\raise.5\dimen7\hbox{\kern.4ex\char'47}}#1\relax \else\if L#1\relax\rlap{\raise.5\dimen7\hbox{\kern.7ex \char'47}}#1\relax\else\message{accent \string\soft \space #1 not defined!}#1\relax\fi\fi\fi\fi\fi\fi}

}
\end{document}